\documentclass[envcountsect, envcountsame]{svmult}

\smartqed

\usepackage{mathptmx}       
\usepackage{helvet}         
\usepackage{courier}        
\usepackage{type1cm}        
\usepackage{makeidx}         
\usepackage{graphicx}        
\usepackage{multicol}        
\usepackage[bottom]{footmisc}

\makeindex             

\usepackage{amstext,amsfonts,amssymb,amscd,amsbsy,amsmath}
\usepackage{tikz-cd}	
\usepackage{tikz}

\newcommand{\CC}{\mathbb{C}}

\newcommand{\PP}{\mathbb{P}}

\DeclareMathOperator{\PGL}{PGL}

\DeclareMathOperator{\mg}{{M}}
\DeclareMathOperator{\mgbar}{\overline{{M}}}

\begin{document}

\title*{Equations of $\,\mgbar_{0,n}$}
\author{Leonid Monin, Julie Rana}
\institute{Leonid Monin \at University of Toronto, 40 St. George Street, Toronto, Ontario, \email{lmonin@math.toronto.edu}
\and Julie Rana \at University of Minnesota, 206 Church St. SE, Minneapolis MN, \email{jrana@umn.edu}}
%
%
\maketitle

\abstract{Following work of Keel and Tevelev, we give explicit polynomials in the Cox ring of $\PP^1\times\cdots\times\PP^{n-3}$ that, conjecturally, determine $\mgbar_{0,n}$ as a subscheme. Using \emph{Macaulay2}, we prove that these equations generate the ideal for $n=5, 6, 7, 8$. For $n \leq 6$ we give a cohomological proof that these polynomials realize $\mgbar_{0,n}$ as a projective variety, embedded in $\mathbb{P}^{(n-2)!-1}$ by the complete log canonical linear system.  
 }

\section{Introduction}
\label{intro}

Let $\mg_{0,n}$ be the moduli space of genus zero curves
with $n$ marked points, and $\mgbar_{0,n}$ its Deligne-Mumford-Knudsen compactification~\cite{deligne-mumford1969}. These spaces were first studied in detail by Knudsen and Mumford~\cite{KM1976, Kn, Kn1983}. 

In \cite{Ka1}, Kapranov constructed $\mgbar_{0,n}$ as a Chow quotient of the Grassmannian $G(2,n)$, which allowed him to present $\mgbar_{0,n}$ as a sequence of blowups of $\PP^{n-3}$.  Following this, Keel and Tevelev in \cite{KeT} described an embedding of $\mgbar_{0,n}$ into the space of sections of particular characteristic classes on $\mgbar_{0,n}$, as well as existence of an embedding $\phi:\mgbar_{0,n}\to\PP^1\times\cdots\times\PP^{n-3}$. 
We obtain a list of equations satisfied by $\phi(\mgbar_{0,n})$ in the Cox ring of $\PP^1\times\cdots\times\PP^{n-3}$.

\begin{lemma}\label{all1}
Consider the embedding $\phi:\mgbar_{0,n}\to \PP^1\times\cdots\times\PP^{n-3}$. Let $w_0^{(i)},\ldots,w_i^{(i)}$ be homogeneous coordinates on the $i^{\textrm{th}}$ component. With a choice of coordinates, the image of $\phi$ satisfies the $\binom{n-1}{4}$ equations given by the $2\times2$ minors of the matrices
$$
\begin{bmatrix}
w_0^{(i)}\left(w_0^{(j)}-w_{i+1}^{(j)}\right)&&&w_1^{(i)}\left(w_1^{(j)}-w_{i+1}^{(j)}\right)&&&\cdots&&&w_i^{(i)}\left(w_i^{(j)}-w_{i+1}^{(j)}\right)\\
w_0^{(j)} &&& w_1^{(j)}&&&\cdots&&&w_i^{(j)}
\end{bmatrix}
$$
for all $1\le i<j\le n-3$.
\end{lemma}

\begin{conjecture}\label{conj:1}
 Let $I_n$ be the prime ideal in the Cox ring of $\PP^1\times\cdots\times\PP^{n-3}$ that defines the embedding $\phi(\mgbar_{0,n})$ scheme-theoretically. Let $J_n$ be the ideal in the Cox ring of $\PP^1\times\cdots\times\PP^{n-3}$ generated by the equations of Lemma~\ref{all1}. Then $I_n$ is the unique $B$-saturation of $J_n$, where $B=\cap_{i=1}^{n-3}\langle w_{0}^{(i)},\ldots, w_{i}^{(i)}\rangle$. The ideal $I_n$ is minimally generated by $\binom{n-1}{d+1}$ polynomials of degree $d$, for $d=3,4,5,\ldots,n-2.$ The degree of $I_n$ is $(2n-7)!!$, the number of trivalent phylogenetic trees on $n-1$ leaves. The lexicographic initial monomial ideals are square-free and Cohen-Macaulay.
\end{conjecture}

We verified Conjecture~\ref{conj:1} for $n=5, 6, 7, 8$ using \emph{Macaulay2}. In particular, we partially answered the question posed in~\cite[Problem~8 on Curves]{Sturmfels}.

In Section~\ref{known}, we provide a list of degree $4$ polynomials contained in $I_n$ and describe a conjectural method to count the number of minimal generators of $I_n$ in any degree. The combinatorial description of the ideal of relations of the invariant ring of $\mgbar_{0,n}$ presented in \cite{HMSV2009} offers another  possibly promising method.

In a slightly different direction, we consider the embedding $\Phi$ of $\mgbar_{0,n}$ into $\PP^{(n-3)!-1}$ given by the $\kappa$ class. By work of Keel and Tevelev~\cite{KeT}, the ideal defining $\mgbar_{0,n}$ as a subscheme of $\PP^{(n-3)!-1}$ is generated by the polynomials defining the Segre embedding $\PP^1\times\cdots\times\PP^{n-3}\to \PP^{(n-3)!-1}$, together with the preimages in the Cox ring of $\PP^{(n-3)!-1}$ of the polynomials in $I_n$ of degree $(2,2,\ldots, 2)$. For example, in the case of $\mgbar_{0,5}$, we have
\begin{corollary}
The ideal of the embedding of $\mgbar_{0,5}$ into $\mathbb{P}^5$ via the $\kappa$ class is generated by the five quadrics
\begin{equation}
\label{eq:fivequadrics}
 t_0 t_1-t_0t_4+t_2 t_3-t_1 t_2,\, \, t_0t_4-t_3 t_4+t_3 t_5-t_1t_5,\,\,
     t_1t_3-t_0 t_4, \,\, t_2t_3-t_0 t_5,\, \, t_2 t_4-t_1t_5.
\end{equation}
  \end{corollary}
  
  Using the Chow quotient description of $\mgbar_{0,n}$ due to Kapranov~\cite{Ka2}, Gibney and Maclagan \cite{GM} provide equations for $\mgbar_{0,5}\subset \PP^{21}$.  Many of the listed equations in~\cite{GM} are linear, so they have effectively given an embedding of $\mgbar_{0,5}$ into $\PP^5$. Using \emph{Macaulay2}, we were able to eliminate variables in such a way that the resulting embedding is a nonsingular variety of dimension 2 given by five quadrics in $\PP^5$. Beyond this, it is not yet clear how these equations relate to ours.
  
The result of Keel and Tevelev motivates the following
\begin{conjecture}\label{kappaconj}
The ideal $J_n$ contains all polynomials of degree $(2,2,\ldots, 2)$ in $I_n$, and is therefore enough to determine the equations of $\Phi(\mgbar_{0,n})$ in the Cox ring of $\PP^{(n-3)!-1}$ ideal-theoretically.
\end{conjecture}
In Section~\ref{tools}, we prove Conjecture~\ref{kappaconj} for $\mgbar_{0,5}$ and $\mgbar_{0,6}$ using cohomological techniques developed in~\cite{KeT}. We hope that these techniques can be extended to larger values of $n$. In particular, we have 
\begin{theorem}
   Let $J_6$ be the ideal generated by the  five polynomials of Lemma~\ref{all}.
Then the ideal $\tilde{I}_6$ generated by polynomials of degree $(d,d, d)$ in $I_6$ is generated by the degree $(2,2,2)$ polynomials in $J_6$. Equivalently, the embedding of $\mgbar_{0,n}$ in $\PP^{23}$ defined by the $\kappa$ class is generated by the polynomials of degree $(2,2,2)$ in $J_6$ and the Segre relations.
  \end{theorem}

The paper is structured as follows. In Section~\ref{background}, we give some background about the moduli spaces $\mgbar_{0,n}$. In Section~\ref{embedding}, we describe in detail the embedding $\phi:\mgbar_{0,n}\to\PP^1\times\cdots\times\PP^{n-3}$ which we will use in later sections to find  equations of $\mgbar_{0,n}$. In particular, away from the boundary of $\mgbar_{0,n}$, we give a parametrization of $\phi$ that extends to the full moduli space, and provide a geometric interpretation of $\phi$ in the cases $n=4$ and $n=5$. In Section~\ref{known}, we give explicit equations satisfied by $\phi(\mgbar_{0,n})$ that are forced by the parametrization. In Sections~\ref{tools} and~\ref{for6} we recall the cohomological machinery developed in~\cite{KeT} and use it to prove that $J_5$ and $J_6$ contain all polynomials of degree $(2,2,\ldots,2)$ in the respective Cox rings.

\section{Background on the moduli space $\mgbar_{0,n}$}\label{background}

We begin with a brief introduction to the moduli space of pointed rational curves. For more details, we recommend the lectures and lecture notes of Cavalieri~\cite{cavalieri2016}. For those new to the theory of stable curves, we recommend~\cite{hm1998}.

For $n \geq 3$, the moduli space $\mg_{0,n}$ parametrizes ordered $n$-tuples of distinct points on $\PP^1$. Two $n$-tuples $(p_1,\ldots,p_n)$ and $(q_1,\ldots, q_n)$ are equivalent if there exists a projective transformation $g \in \PGL(2,\CC)$ such that
$$
(q_1,\ldots, q_n) = (g(p_1),\ldots,g(p_n)).
$$ 
Since a projective transformation can map three points in $\PP^1$ to any other three points and is uniquely determined by their image,  the dimension of  $\mg_{0,n}$ equals  $n-3$.

The space $\mg_{0,n}$ is not compact because the points $p_i$ are distinct. There are a number of compactifications of $\mg_{0,n}$, including those described by Losev-Manin~\cite{losevmanin2000} and Keel~\cite{keel1992}. But the first and most well-known is $\mgbar_{0,n}$, the Deligne-Mumford-Knudsen compactification, described explicitly by Kapranov~\cite{Ka2,Ka1}. The moduli space~$\mgbar_{0,n}$ parametrizes \emph{stable $n$-pointed rational curves}.
\begin{definition}
 A stable  $n$-pointed rational curve is a tuple $(C, p_1, \ldots, p_n)$, where
 \begin{enumerate}
 \item $C$ is a connected curve of arithmetic genus $0$ with at most simple nodal singularities;
 \item $p_1, \ldots, p_n$ are distinct nonsingular points on $C$;
 \item each irreducible component of $C$ has at least three special points (either marked points or nodes).
 \end{enumerate}
\end{definition}

The {\em dual graph} of a stable curve $(C, p_1, \ldots, p_n)$ is defined as a graph
having  a vertex for each irreducible component, an edge between two vertices  for each point of intersection of the corresponding components, and for each marked point, a labeled half edge attached to the appropriate vertex. Since $C$ has arithmetic genus $0$, the dual graph is a tree.

The boundary $\mgbar_{0,n} \setminus \mg_{0,n}$ is a normal crossing divisor with a natural stratification by the dual graphs. The codimension of the stratum $\delta(\Gamma)$ in $\mgbar_{0,n}$ corresponding to the dual graph $\Gamma$ is one less than the number of vertices of $\Gamma$: 
 $$ 
 \textrm{codim}(\delta(\Gamma)) = \# V(\Gamma) -1.
 $$
Thus, each divisorial component of the boundary of $\mgbar_{0,n}$ corresponds to stable curves with dual graph $\Gamma$ corresponding to a partition of $\{1,\ldots,n\}$ into two disjoint sets $I$ and $I^c$, each of cardinality at least 2. Given such a partitition, we denote the corresponding irreducible boundary divisor of $\mgbar_{0,n}$ by $\delta_I$.
Note that two divisors $\delta_I$ and $\delta_{J}$ intersect in
$\mgbar_{0,n}$ if and only if $I\subset J$, $I\subset J^c$, $J\subset I$, or $J\subset I^c$.

\begin{example} Consider the graph whose vertices correspond to the irreducible boundary divisors in $\mgbar_{0,5}$. Two vertices are joined by an edge if the corresponding boundary divisors intersect (see Figure~\ref{petersen}). Originally a purely combinatorial construction, the Petersen graph shown in Figure~\ref{petersen} is in fact the link of the tropical moduli space $M^{\textrm{trop}}_{0,5}$. For a brief introduction to the moduli space $M^{\textrm{trop}}_{g,n}$, we recommend the lectures and lecture notes of Melody Chan~\cite{chan2016}. For a thorough introduction to tropical geometry, see~\cite{MS2015}.
\begin{figure}
\center
\begin{tikzpicture}[style=thick,scale=1]
\draw (18:2cm) -- (90:2cm) -- (162:2cm) -- (234:2cm) --
(306:2cm) -- cycle;
\draw (18:1cm) -- (162:1cm) -- (306:1cm) -- (90:1cm) --
(234:1cm) -- cycle;
\foreach \x in {18,90,162,234,306}{
\draw (\x:1cm) -- (\x:2cm);
\draw[black,fill=black] (\x:2cm) circle (2pt);
\draw[black,fill=black] (\x:1cm) circle (2pt);
}
\node[label=$\delta_{1,2}$] at (18:2cm) {}; 
\node[label=$\delta_{3,5}$] at (90:2cm) {}; 
\node[label=$\delta_{1,4}$] at (162:2cm) {}; 
\node[label=left:$\delta_{2,5}$] at (234:2cm) {}; 
\node[label=right:$\delta_{3,4}$] at (306:2cm) {}; 
\node[label=above:$\delta_{4,5}$] at (18:1cm) {}; 
\node[label=right:$\delta_{2,4}$] at (90:1cm) {}; 
\node[label=$\delta_{2,3}$] at (162:1cm) {}; 
\node[label=left:$\delta_{1,3}$] at (234:1cm) {}; 
\node[label=right:$\delta_{1,5}$] at (306:1cm) {}; 
\end{tikzpicture}
\caption{The Petersen graph describes the boundary complex of $\mgbar_{0,5}$.}\label{petersen}
\end{figure}

\end{example}

\begin{remark}
A general philosophy is that birational models of a given compactified moduli space should provide alternate compactifications which themselves have modular interpretations. From this perspective,  describing the birational geometry of $\mgbar_{0,n}$ is not only interesting in its own right, but also provides a window into current research in moduli theory. 
As a first step in this direction, in~\cite{harrismumford82}, Harris and Mumford proved that the moduli spaces $\mgbar_{g,n}$ are of general type for large enough $g$. Thus, understanding the birational geometry of $\mgbar_{g,n}$ boils down to describing all ample divisors on the moduli space. A long-standing conjecture of Fulton and Faber, the so-called ``F-conjecture'', describes the ample cone of $\mgbar_{g,n}$; see for example~\cite{gibneykeelmorrison2002} in which the F-conjecture is reduced to the genus $0$ case. Notably, in~\cite{hukeel2000}, Hu and Keel conjectured that $\mgbar_{0,n}$ is a Mori Dream Space, a result that would have implied the F-conjecture. This was recently disproved for $n>133$ by Castravet and Tevelev~\cite{castravet-tevelev2015}; their techniques were quickly extended to $n>13$ in~\cite{gonzalez-karu2014}. 
We recommend the excellent survey~\cite{fedorchuksmyth2011} for those interested in learning more about birational models and alternative compactifications of $\mgbar_{g,n}$.
\end{remark}

\section{The embedding $\mgbar_{0,n}$ in $\PP^1\times\cdots\times\PP^{n-3}$}\label{embedding}

We begin with a description of an embedding $\phi:\mgbar_{0,n}\hookrightarrow\PP^1\times\cdots\times\PP^{n-3}$, followed by a parametrization of $\phi$ on the interior of $\mgbar_{0,n}$ which extends to the full moduli space.  We then give a geometric interpretation of $\phi$ in the cases $\mgbar_{0,4}$ and $\mgbar_{0,5}$. This realizes $\mgbar_{0,5}$ in Theorem~\ref{m05}
 as a pencil of conics in $\PP^1\times\PP^2$.
 Finally,  we use the parametrization to list a set of equations satisfied by $\phi(\mgbar_{0,n})\subset\PP^1\times\cdots\times\PP^{n-3}$.

We recall two well-studied maps from $\mgbar_{0,n}$. The first, the forgetful map $\pi_n:\mgbar_{0,n}\to\mgbar_{0,n-1}$, is given by forgetting the last point of $[C,p_1,\ldots p_n]$  and stabilizing the curve. That is, $\pi_n$ contracts the components of $C$ that have less than three special points among $p_1,\ldots, p_{n-1}$, and remembers the points of intersection.

For the second, let $\mathbb{L}_i$ be the line bundle on $\mgbar_{0,n}$ whose fiber over a point $[C,p_1,\ldots,p_n]$ is the cotangent space of $\PP^1$ at $p_i$. Define $\psi_i=c_1(\mathbb{L}_i)$ to be the first Chern class of $\mathbb{L}_i$. The Kapranov map $\psi_n$ is the rational map given by the linear system $|\psi_n|$. This map was first described in detail by Kapranov~\cite{Ka1}, who proved in particular that  $\psi_n:\mgbar_{0,n}\to\PP^{n-3}$.

\begin{theorem}\label{thm:embedding}[KT, Cor 2.7] The map $\Phi=(\pi_n,\psi_n):\mgbar_{0,n}\rightarrow \mgbar_{0,n-1}\times \PP^{n-3}$ is a closed embedding.
\end{theorem}

\begin{corollary}\label{cor:embedding} We have a closed embedding $\phi:\mgbar_{0,n}\hookrightarrow \PP^1\times\PP^2\times\ldots \times\PP^{n-3}$.
\end{corollary}

\begin{proof}
Apply Theorem~\ref{thm:embedding} successively. \qed
\end{proof}  

Let us describe the map $\phi$ of Corollary~\ref{cor:embedding} explicitly. Since $\phi$ is a closed embedding, it is enough to describe it only on the smooth part $\mg_{0,n}$ of $\mgbar_{0,n}$, which is an open, dense subset of $\mgbar_{0,n}$.
To this end, consider the restriction of $\pi_n$ to $\mg_{0,n}$ and let $F=\pi_n^{-1}([\PP^1,p_1,\ldots, p_{n-1}])\simeq\PP^1\backslash\{p_1,\ldots,p_{n-1}\}$ be the fiber over a point in $\mg_{0,n-1}$.  We have the following description of the line bundle $\mathbb{L}_n$ over the fiber $F$.

\begin{lemma}\label{psin}\cite{KeT} We have $\mathbb{L}_n|_F=\omega_{\PP^1}(p_1+\cdots + p_{n-1})$. In particular, $\psi_n|_F=K_{\PP^1}+p_1+\cdots+p_{n-1}.$
\end{lemma}

Following  Lemma~\ref{psin}, we obtain a basis of $H^0(F, K_{F}+p_1+\cdots+p_{n-1})$. 
\begin{lemma}\label{basis} The vector space $H^0(F, K_{F}+p_1+\cdots+p_{n-1})$ has dimension $n-2$. A basis is given by the one-forms
$$
\left\{\frac{dx}{(x-p_{1})(x-p_2)},\ldots,\frac{dx}{(x-p_{1})(x-p_{n-1})} \right\}.
$$
\end{lemma}
\begin{proof}
Since the canonical class of $F\simeq\PP^1$ is $K_F= -2[\textrm{pt}]$, the dimension of the space of global sections follows from, for example, Riemann-Roch. Since the $n-2$ one-forms listed have two poles, each at different pairs of points, they are linearly independent, and so form a basis. \qed
\end{proof} 

We thus obtain an explicit parametrization of the map $\psi_n|_F:F\hookrightarrow \PP^{n-3}$.
\begin{proposition}\label{thm:psi} Let $
F\simeq \PP^1\backslash\{p_1,\ldots, p_{n-1}\} $
be the fiber of the map  $\pi_n$
over the point $[\PP^1,p_1,\ldots,p_{n-1}]\in\mg_{0,n-1}$. 
The restriction of $\psi_n$ to $F$ is given in coordinates by
$$
\psi_n|_F \,: \, F\hookrightarrow \PP^{n-3} \,\,,\quad
x\mapsto \left[\frac{p_1-p_2}{x-p_2}:\cdots:\frac{p_1-p_{n-1}}{x-p_{n-1}}\right].
$$
\end{proposition}

\begin{proof}
Using a properly rescaled basis from Lemma~\ref{basis}, we can write the map~as 
$$
x\mapsto\left[\frac{p_1-p_2}{(x-p_{1})(x-p_2)}:\ldots:\frac{p_1-p_{n-1}}{(x-p_{1})(x-p_{n-1})}\right].
$$
Multiplying through by $x-p_1$ gives the result.  \qed 
\end{proof}

By Corollary~\ref{cor:embedding}, we can extend $\psi_n$ uniquely to all of $\mgbar_{0,n}$. By our choice of basis, the points $p_1,\ldots,p_{n-1}\in \bar{F}$ map to the coordinate points of $\PP^{n-3}$, in particular to points in general position. Thus, $\psi_n(\bar{F})$ is a degree $n-3$ curve in $\PP^{n-3}$, i.e. a (generically) smooth rational normal curve passing through these $n-1$ fixed points. 

Taking the parameter $x$ to be $p_n$, the map $\phi:\mgbar_{0,n}\rightarrow \PP^1\times\cdots\times\PP^{n-3}$
is given in the $i^{\textrm{th}}$ component by
$$
[C,p_1,\ldots,p_n] \mapsto \ \left[\frac{p_1-p_2}{p_{i+3}-p_2}:\cdots:\frac{p_1-p_{i+2}}{p_{i+3}-p_{i+2}}\right].
$$


\begin{example}
We describe the embedding $\phi:\mgbar_{0,4}\to\PP^1$ explicitly. Away from $p_2, p_3, p_4$, we have
$$
\phi: [\PP^1,p_1,\ldots,p_4] \mapsto \ \left[\frac{p_1-p_2}{p_4-p_2}:\frac{p_1-p_{3}}{p_4-p_{3}}\right].
$$
Thus, away from the boundary of $\mgbar_{0,4}$, we see that $\phi$ is an isomorphism mapping a 4-tuple of distinct points on $\PP^1$ to their cross-ratio. 

The boundary of $\mgbar_{0,4}$ consists of three points: $\delta_{1,2}$, $\delta_{1,3}$, and $\delta_{1,4}$ (see Fig.~\ref{m04}). By taking limits $p_1\to p_2$, $p_1\to p_3$, and $p_1\to p_4$, respectively, we see that these boundary points map under $\phi$ to the points $[0:1],[1:0],$ and $[1:1]$, respectively.

\begin{figure}
\begin{center}
\begin{tikzpicture}
\draw (0,0) -- (1.2,1.2);
\draw (1,1.2)--(2.2,0);
\draw (.35,.25)--(.25,.35);
\draw (.75,.65)--(.65,.75);
\node [above] at (0,.3) {$p_1$};
\node [above] at (.4,.7) {$p_2$};
\draw (1.95,.35)--(1.85,.25);
\draw (1.45,.65)--(1.55,.75);
\node [above] at (1.8,.7) {$p_3$};
\node [above] at (2.2,.3) {$p_4$};

\node [below] at (1.1,0) {$\delta_{1,2}$};

\draw (3,0) -- (4.2,1.2);
\draw (4,1.2)--(5.2,0);
\draw (3.35,.25)--(3.25,.35);
\draw (3.75,.65)--(3.65,.75);
\node [above] at (3,.3) {$p_1$};
\node [above] at (3.4,.7) {$p_3$};
\draw (4.95,.35)--(4.85,.25);
\draw (4.45,.65)--(4.55,.75);
\node [above] at (4.8,.7) {$p_2$};
\node [above] at (5.2,.3) {$p_4$};

\node [below] at (4.1,0) {$\delta_{1,3}$};

\draw (6,0) -- (7.2,1.2);
\draw (7,1.2)--(8.2,0);
\draw (6.35,.25)--(6.25,.35);
\draw (6.75,.65)--(6.65,.75);
\node [above] at (6,.3) {$p_1$};
\node [above] at (6.4,.7) {$p_4$};
\draw (7.95,.35)--(7.85,.25);
\draw (7.45,.65)--(7.55,.75);
\node [above] at (7.8,.7) {$p_2$};
\node [above] at (8.2,.3) {$p_3$};

\node [below] at (7.1,0) {$\delta_{1,4}$};
\end{tikzpicture}
\caption{The  boundary divisors of $\mgbar_{0,4}$.}\label{m04}
\end{center}
\end{figure}
\end{example}

\begin{example}
For $n=5$ the embedding above has the form
$$
\phi:\mgbar_{0,5}\rightarrow \PP^1\times\PP^{2}\simeq\mgbar_{0,4}\times\PP^{2}
$$
$$
[\PP^1,p_1,\ldots,p_5] \mapsto \ \left(\left[\frac{p_1-p_2}{p_4-p_2}:\frac{p_1-p_{3}}{p_4-p_{3}}\right],\left[\frac{p_1-p_2}{p_5-p_2}:\frac{p_1-p_{3}}{p_5-p_{3}}:\frac{p_1-p_{4}}{p_5-p_{4}}\right]\right) .
$$
The forgetful map $\pi_5$ restricts to $\mgbar_{0,5}$ the projection of $\mgbar_{0,4}\times\PP^{2}$ onto the first factor, so the fiber of $\pi_5$ over any point  in $\mg_{0,4}\simeq \PP^1 \setminus \{[0:1],[1:0],[1:1]\}$  is a smooth conic passing through four fixed points $[1:0:0]$,$[0:1:0]$,$[0:0:1]$, and $[1:1:1]$ in a copy of $\PP^2$. The fibers over $[0:1],[1:0],[1:1]$ are the singular conics passing through these four fixed points.
In particular, six boundary divisors of $\mgbar_{0,5}$ map to the components of three singular conics, and the remaining four map to the fibers $\PP^1\times\{[1:0:0]\}$, $\PP^1\times\{[0:1:0]\}$, $\PP^1\times\{[0:0:1]\}$, and $\PP^1\times\{[1:1:1]\}$ (see Fig.~\ref{m05pic}).

\begin{figure}
\begin{center}
\begin{tikzpicture}
\draw (0,0) -- (0,5);
\draw (-.1,0.7)--(.1,0.7);
\draw (-.1,2.7)--(.1,2.7);
\draw (-.1,4.7)--(.1,4.7);
\node [left] at (-.1,.7) {[0:1]};
\node [left] at (-.1,2.7) {[1:1]};
\node [left] at (-.1,4.7) {$[1:0]$};
\node [below] at (0,0) {$\PP^1$};
\end{tikzpicture}
\includegraphics[scale=.3]{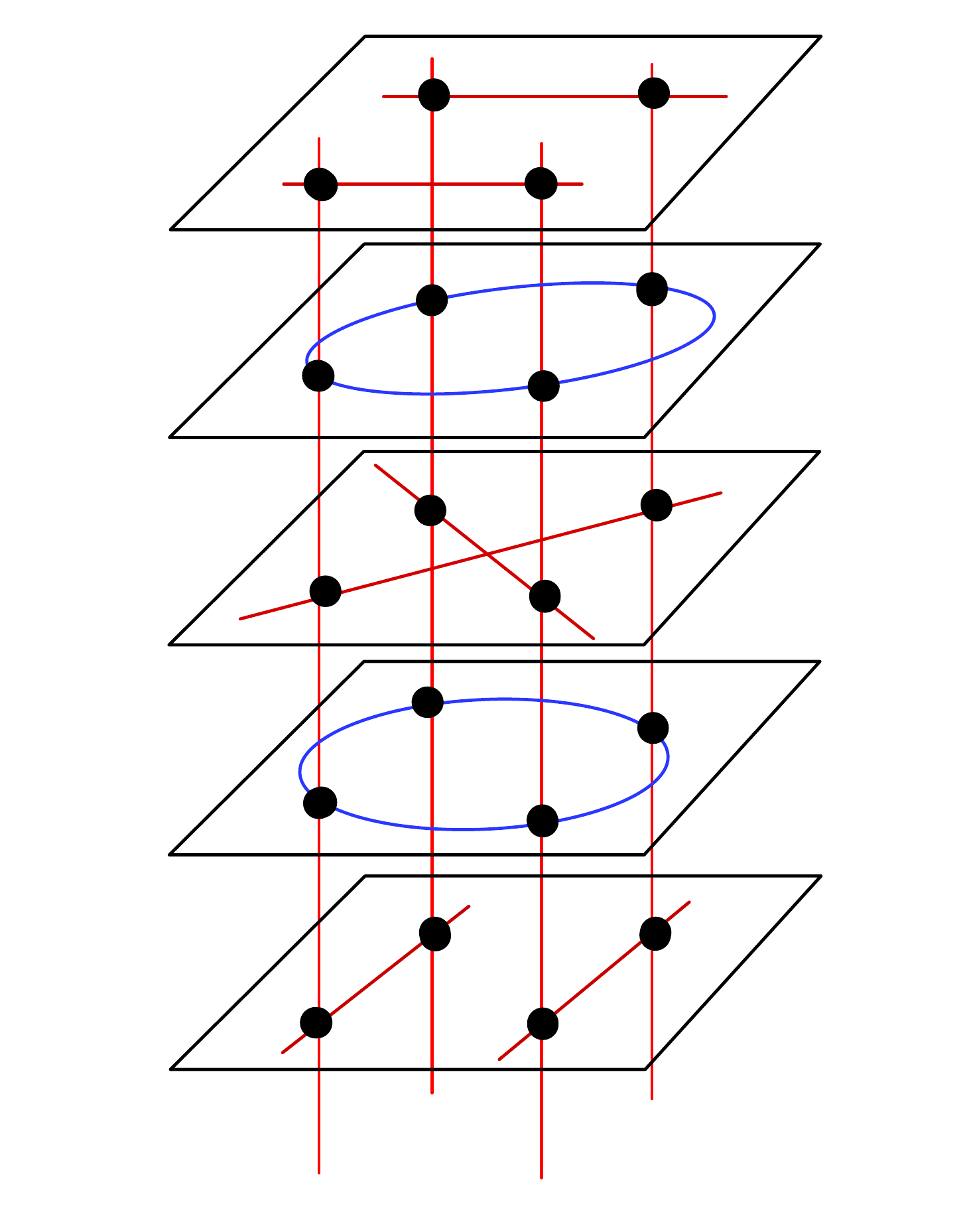}
\caption{The embedding $\phi:\mgbar_{0,5}\to\PP^1\times\PP^2$. The red lines correspond to boundary divisors.}\label{m05pic}
\end{center}
\end{figure}

Since $\mgbar_{0,5}$ is the pencil of conics in $\PP^2$ passing through these given four points, we easily write down an embedding of $\mgbar_{0,5}$ in $\PP^1_{[a_0:a_1]}\times\PP^{2}_{[b_0:b_1:b_2]}$. Its
equation is
\begin{equation}
\label{eq:abeq}
a_0 b_1(b_0-b_2)-a_1b_0(b_1-b_2).
\end{equation}

We observe that $\psi: \mgbar_{0,5} \to \PP^2$ is a blowdown. Indeed, since there is unique conic passing through any five points in $\PP^2$, the map $\psi: \mgbar_{0,5} \to \PP^2$ is 1-to-1 away from $[1:0:0],[0:1:0],[0:0:1]$, and $[1:1:1]$. The fibers over these points are all isomorphic to $\PP^1$.

\begin{proposition}\label{m05}
 With the choice of coordinates above, the ideal which defines the moduli space $\mgbar_{0,5}$ as a projective variety in $\PP^1_{[a_0:a_1]}\times\PP^2_{[b_0:b_1:b_2]}$ is generated by
 (\ref{eq:abeq}).
 \end{proposition}
 
\begin{proof}
This is a direct computation consisting of checking that the variety described by this 
principal ideal in the multigraded ring $\mathbb{C}[a_0,a_1,b_0,b_1,b_2]$ is two-dimensional, smooth, and irreducible. We verified this using \emph{Macaulay2}. \qed
\end{proof}

\end{example}



\section{Equations of $\mgbar_{0,n}$}\label{known}

We provide a list of equations contained in the ideal defining the scheme $\phi(\mgbar_{0,n})$ in the Cox ring of $\PP^1\times\cdots\times\PP^{n-3}$.
Let $w_0^{(i)},\ldots,w_{i}^{(i)}$ be homogeneous coordinates on the $i^{\textrm{th}}$ 
factor of $ \PP^1\times\cdots\times\PP^{n-3}$. By Proposition~\ref{thm:psi}, the embedding $\phi:\mgbar_{0,n}\rightarrow \PP^1\times\cdots\times\PP^{n-3}$ is given in coordinates by
$$
w_{j}^{(i)}=\frac{p_{1,j+2}}{p_{i+3,j+2}},
\quad \hbox{where $p_{i,j}=p_{i}-p_{j}$.}
$$

\begin{lemma}\label{all}
The image of $\phi$ in  $\PP^1\times\cdots\times\PP^{n-3}$ satisfies the $\binom{n-1}{4}$ polynomials given by the $2\times2$ minors of the matrices
$$
\begin{bmatrix}
w_0^{(i)}\left(w_0^{(j)}-w_{i+1}^{(j)}\right)&&&w_1^{(i)}\left(w_1^{(j)}-w_{i+1}^{(j)}\right)&&&\cdots&&&w_i^{(i)}\left(w_i^{(j)}-w_{i+1}^{(j)}\right)\\
w_0^{(j)} &&& w_1^{(j)}&&&\cdots&&&w_i^{(j)}
\end{bmatrix}
$$
for all $1\le i<j\le n-3$. Let $J_n$ be the ideal generated by these polynomials. No proper subset of these polynomials forms a basis of $J_n$. 
\end{lemma}

\begin{proof}
The proof is direct calculation. Choose columns $r$ and $s$. We show that 
$$
\left(w_{r}^{(j)}-w_{i+1}^{(j)}\right)w_{s}^{(j)}w_{r}^{(i)}=\left(w_{s}^{(j)}-w_{i+1}^{(j)}\right)w_{r}^{(j)}w_{s}^{(i)}.
$$
Indeed, the following equalities hold:
\begin{eqnarray*}
w_{r}^{(j)}-w_{i+1}^{(j)}&=&\,\frac{p_{1,r+2}}{p_{j+3,r+2}}-\frac{p_{1,i+3}}{p_{j+3,r+3}}  \\
&=&\, \frac{p_{1,r+2}p_{j+3,i+3}-p_{1,i+3}p_{j+3,r+2}}{p_{j+3,r+2}p_{j+3,i+3}} 
\,\,=\,\,
\frac{p_{1,j+3} p_{r+2,i+3}}{p_{j+3,r+2}p_{j+3,i+3}}.
\end{eqnarray*}
So we have
\begin{eqnarray*}
\left(w_{r}^{(j)}-w_{i+1}^{(j)}\right)w_{s}^{(j)}w_{r}^{(i)}
&=&\frac{p_{1,j+3} p_{r+2,i+3}}{p_{j+3,r+2}p_{j+3,r+2}}\cdot\frac{p_{1,s+2}}{p_{j+3,s+2}}\cdot\frac{p_{1,r+2}}{p_{i+3,r+2}}\\
&=&-\frac{p_{1,j+3}p_{1,s+2} p_{1,r+2}}{p_{j+3,r+2}p_{j+3,s+2}p_{j+3,i+3}}.
\end{eqnarray*}
A similar calculation gives that $\left(w_{s}^{(j)}-w_{i+1}^{(j)}\right)w_{r}^{(j)}w_{s}^{(i)}$ 
equals the same expression.

For  each $i\in\{1,\ldots,n-4\}$ there are $n-3-i$ matrices with $i+1$ columns each. Thus, the number of $2\times 2$ minors is given by
$$\sum_{i=1}^{n-4}(n-3-i)\binom{i+1}{2} \,\,\, = \,\,\, \binom{n-1}{4}. $$
The polynomials
 have the same total degree and different initial terms under lex order, so are linearly independent over $\mathbb{C}$. This proves the final statement. 
\qed
\end{proof} 


\begin{example}
We list the polynomials of Lemma~\ref{all} satisfied by $\mgbar_{0,7}$ as a subscheme of $\PP^1\times\PP^2\times\PP^3\times\PP^4$. The Cox ring of the product of projective spaces~is
$$\mathbb{C}[a_0,a_1,b_0,b_1,b_2,c_0,c_1,c_2,c_3,d_0,d_1,d_2,d_3,d_4]. $$ 
The six matrices of Lemma~\ref{all} are
\begin{small}
$$ \begin{bmatrix}
a_0(b_0-b_2) & a_1(b_1-b_2)\\
b_0 &b_1
\end{bmatrix}, 
\begin{bmatrix}
a_0(c_0-c_2) & a_1(c_1-c_2)\\
c_0 &c_1
\end{bmatrix},
\begin{bmatrix}
a_0(d_0-d_2) & a_1(d_1-d_2)\\
d_0 &d_1
\end{bmatrix}
$$
$$
\begin{bmatrix}
b_0(c_0-c_3) & b_1(c_1-c_3) &b_2(c_2-c_3)\\
c_0 &c_1 &c_2
\end{bmatrix},
\begin{bmatrix}
b_0(d_0-d_3) & b_1(d_1-d_3) &b_2(d_2-d_3)\\
d_0 &d_1 &d_2
\end{bmatrix}$$
$$
\begin{bmatrix}
c_0(d_0-d_4) & c_1(d_1-d_4) &c_2(d_2-d_4)&c_3(d_3-d_4)\\
d_0 &d_1 &d_2&d_3
\end{bmatrix}.
$$
\end{small}

Taking all $2\times 2$ minors gives the following $15$ polynomials
in the ideal defining $\mgbar_{0,7}$ as a subscheme of $\PP^1\times\PP^2\times \PP^3\times\PP^4$:
$$ \begin{matrix}
c_2d_2d_3-c_3d_2d_3+c_3d_2d_4-c_2d_3d_4,\,\,
 c_1d_1d_3-c_3d_1d_3+c_3d_1d_4-c_1d_3d_4 \\
c_0d_0d_3-c_3d_0d_3+c_3d_0d_4-c_0d_3d_4,\,\,
 c_1d_1d_2-c_2d_1d_2+c_2d_1d_4-c_1d_2d_4 \\
b_1d_1d_2-b_2d_1d_2+b_2d_1d_3-b_1d_2d_3, \,\,
c_0d_0d_2-c_2d_0d_2+c_2d_0d_4-c_0d_2d_4 \\
b_0d_0d_2-b_2d_0d_2+b_2d_0d_3-b_0d_2d_3, \,\,
c_0d_0d_1-c_1d_0d_1+c_1d_0d_4-c_0d_1d_4 \\
b_0d_0d_1-b_1d_0d_1+b_1d_0d_3-b_0d_1d_3, \,\,
a_0d_0d_1-a_1d_0d_1+a_1d_0d_2-a_0d_1d_2 \\
b_1c_1c_2-b_2c_1c_2+b_2c_1c_3-b_1c_2c_3, \,\,
b_0c_0c_2-b_2c_0c_2+b_2c_0c_3-b_0c_2c_3 \\
b_0c_0c_1-b_1c_0c_1+b_1c_0c_3-b_0c_1c_3, \,\,
a_0c_0c_1-a_1c_0c_1+a_1c_0c_2-a_0c_1c_2 \\
a_0b_0b_1-a_1b_0b_1+a_1b_0b_2-a_0b_1b_2
\end{matrix}
$$
\end{example}

\begin{conjecture}\label{simple2} 
Let $I_n$ be the $B$-saturated ideal in $\mathbb{C}[\{w_j^{(i)}\}]$ that defines the subscheme $\phi(\mgbar_{0,n})$ of $\PP^1\times\cdots \times \PP^{n-3}$, where $B=\cap_{i=1}^{n-3}\langle w_{0}^{(i)},\ldots, w_{i}^{(i)} \rangle$. Let $G$ be a minimal Gr\"obner basis of  $I_n$. The number of polynomials of degree $d$ in $G$ is $\binom{n-1}{d+1}$.
\end{conjecture}

Computations in \emph{Macaulay2} support Conjecture~\ref{simple2} for $4\le n\le 8$. The following lemma gives additional evidence.

\begin{lemma}\label{degree4}
For each choice of $0\le i<j\le k< l<m\leq n-3$, the embedding of $\mgbar_{0,n}$ in $\PP^1\times\cdots\times \PP^{n-3}$ satisfies the following $\binom{n-1}{5}$ linearly independent equations of degree 4 in the Cox ring:
\begin{equation}
\begin{split}\label{deg4}
w_{k+1}^{(m)}w_{l+1}^{(m)} \left(w_{i}^{(k)}w_{j}^{(l)} -w_{j}^{(k)}w_{i}^{(l)} \right)+ 
w_{l+1}^{(m)}w_{j}^{(m)}\left(w_{j}^{(k)}w_{i}^{(l)} - w_{i}^{(k)}w_{i}^{(l)} \right) + \\
+ w_{j}^{(m)}w_{k+1}^{(m)} \left(w_{i}^{(k)}w_{i}^{(l)} - w_{i}^{(k)}w_{j}^{(l)}\right)=0.
\end{split}
\end{equation}
\end{lemma}

\begin{proof}
Note that the number of equations with $j=k$ is $\binom{n-2}{5}$ and with $j\neq k$ is $\binom{n-2}{4}$. Thus there are $\binom{n-2}{5}+\binom{n-2}{4}=\binom{n-1}{5}$ equations. Since all of the equations have different initial terms with respect to lex order, they are linearly independent.

The rest of the proof is direct calculation. For simplicity, we first shift the indices:
$$
w_{k-2}^{(m-3)}w_{l-2}^{(m-3)} \left(w_{i-2}^{(k-3)}w_{j-2}^{(l-3)} -w_{j-2}^{(k-3)}w_{i-2}^{(l-3)} \right)+ $$
$$+w_{l-2}^{(m-3)}w_{j-2}^{(m-3)}\left(w_{j-2}^{(k-3)}w_{i-2}^{(l-3)} - w_{i-2}^{(k-3)}w_{i-2}^{(l-3)} \right) + 
$$
$$
+w_{j-2}^{(m-3)}w_{k-2}^{(m-3)} \left(w_{i-2}^{(k-3)}w_{i-2}^{(l-3)} - w_{i-2}^{(k-3)}w_{j-2}^{(l-3)}\right)=0.
$$
Then one checks that the following identities hold:
$$
w_{k-2}^{(m-3)}w_{l-2}^{(m-3)} \left(w_{i-2}^{(k-3)}w_{j-2}^{(l-3)} -w_{j-2}^{(k-3)}w_{i-2}^{(l-3)} \right)=
A\frac{p_{k,l}}{p_{m,l}p_{m,k}p_{k,j}p_{l,j}};
$$
$$
w_{l-2}^{(m-3)}w_{j-2}^{(m-3)}\left(w_{j-2}^{(k-3)}w_{i-2}^{(l-3)} - w_{i-2}^{(k-3)}w_{i-2}^{(l-3)} \right) = 
A\frac{1}{p_{m,l}p_{m,j}p_{k,j}};
$$
$$
w_{j-2}^{(m-3)}w_{k-2}^{(m-3)} \left(w_{i-2}^{(k-3)}w_{i-2}^{(l-3)} - w_{i-2}^{(k-3)}w_{j-2}^{(l-3)}\right)=
-A\frac{1}{p_{m,j}p_{m,k}p_{l,j}};
$$
where $A=\frac{p_{1,i}p_{1,j}p_{1,l}p_{1,k}p_{j,i}}{p_{l,i}p_{k,i}}$. Furthermore the following is true: 
\begin{align*}
\frac{1}{p_{m,j}p_{m,k}p_{l,j}}-\frac{1}{p_{m,l}p_{m,j}p_{k,j}}&=\frac{1}{p_{m,j}}\frac{p_{m,l}p_{k,j}-p_{m,k}p_{l,j}}{p_{m,l}p_{k,j}p_{m,k}p_{l,j}}\\
=\quad \frac{p_{m,j}}{p_{m,j}}\frac{p_{k,l}}{p_{m,l}p_{m,k}p_{k,j}p_{l,j}} \,
&=\frac{p_{k,l}}{p_{m,l}p_{m,k}p_{k,j}p_{l,j}},
\end{align*}
which completes the proof. 
\qed
\end{proof}

The equations in Lemma~\ref{degree4} are equivalent to the following, modulo the ideal~$J_n$:
\begin{equation}
\begin{split}\label{deg4'}
w^{(k)}_{i}w^{(l)}_{i}w^{(m)}_{j}w^{(m)}_{k+1} - w^{(k)}_{i}w^{(l)}_{k+1}w^{(m)}_{j}w^{(m)}_{k+1} - w^{(k)}_{i}w^{(l)}_{i}w^{(m)}_{j}w^{(m)}_{l+1} +\\
+w^{(k)}_{j}w^{(l)}_{i}w^{(m)}_{j}w^{(m)}_{l+1} + w^{(k)}_{i}w^{(l)}_{k+1}w^{(m)}_{j}w^{(m)}_{l+1} - w^{(k)}_{j}w^{(l)}_{i}w^{(m)}_{k+1}w^{(m)}_{l+1}.
\end{split}
\end{equation}
Indeed, we have
$$
(\ref{deg4})-(\ref{deg4'}) = w^{(k)}_i\left[ w_{k+1}^{(l)}w_{j}^{(m)}\left(w_{k+1}^{(m)} - w_{l+1}^{(m)} \right) - w_{j}^{(l)}w_{k+1}^{(m)}\left(w_{j}^{(m)} - w_{l+1}^{(m)} \right)\right]\in J_n.
$$ 
We expect that the polynomials~\eqref{deg4'} are the minimal generators of $I_n$ in degree $4$.

\begin{example}
According to Conjecture~\ref{simple2}, we expect $\binom{7-1}{4+1}=6$ polynomials of degree $4$ in a reduced Gr\"obner basis of $I_7$. Using \emph{Macaulay2}, we compute the unique $B$-saturation of $J_7$. This gives six polynomials of degree $4$ in
$$\mathbb{C}[a_0,a_1,b_0,b_1,b_2,c_0,c_1,c_2,c_3,d_0,d_1,d_2,d_3,d_4],$$
in support of Conjecture~\ref{simple2}. These are
$$ \begin{matrix}
b_1c_1d_2d_3-b_1c_3d_2d_3-b_1c_1d_2d_4+b_2c_1d_2d_4+b_1c_3d_2d_4-b_2c_1d_3d_4 \\
b_0c_0d_2d_3-b_0c_3d_2d_3-b_0c_0d_2d_4+b_2c_0d_2d_4+b_0c_3d_2d_4-b_2c_0d_3d_4 \\
b_0c_0d_1d_3-b_0c_3d_1d_3-b_0c_0d_1d_4+b_1c_0d_1d_4+b_0c_3d_1d_4-b_1c_0d_3d_4 \\
a_0c_0d_1d_2-a_0c_2d_1d_2-a_0c_0d_1d_4+a_1c_0d_1d_4+a_0c_2d_1d_4-a_1c_0d_2d_4 \\
a_0b_0d_1d_2-a_0b_2d_1d_2-a_0b_0d_1d_3+a_1b_0d_1d_3+a_0b_2d_1d_3-a_1b_0d_2d_3 \\
a_0b_0c_1c_2-a_0b_2c_1c_2-a_0b_0c_1c_3+a_1b_0c_1c_3+a_0b_2c_1c_3-a_1b_0c_2c_3.
\end{matrix} $$
These six polynomials correspond to the equations~\eqref{deg4'}.

We note two interesting features. First, the last equation coincides with the unique equation of degree $4$ in a Gr\"obner basis of $I_6$, and the structure of each equation is similar to this one. In general, we expect that equations of degree $d$ in a Gr\"obner basis for $I_n$ have structure similar to the unique equation of top degree in a Gr\"obner basis for $I_{d+2}$.
Secondly, for each pair $0\le i<j\le 2$ there is a unique polynomial of degree $(0,1,1,2)$ containing $b_i$ and $b_j$. 
\end{example}

Let $G$ be a reduced Gr\"obner basis of $I_n$ under lex order, and  $G_d$ the subset of $G$ consisting of polynomials of degree $d$.  Fix a $\PP^i$ and choose two $w_{j}^{(i)}, w_{k}^{(i)}$. For each choice of $d-2$ of the remaining spaces $\PP^{i+1},\ldots,\PP^{n-3}$, we conjecture that there is a unique polynomial in $G_d$ of degree one in the chosen variables, other than the last occurring projective space, in which the polynomial has degree two. 

If the polynomials of degree $d$ can be counted in this way, then Conjecture~\ref{simple2} will be true, upon application of the following combinatorial fact:

\begin{lemma}\label{nifty}
$$\sum_{i=1}^{n-3}\binom{n-3-i}{d-2}\binom{i+1}{2}=\binom{n-1}{d+1}$$
\end{lemma}
\begin{proof}
We rewrite the left hand side as a hypergeometric series and apply the Chu-Vandermonde Identity, see e.g.~\cite{Roy1987}.  Let ${}_2F_1\begin{pmatrix}a&b\\ & c\end{pmatrix}$ denote the series
$$\sum_{k=0}^{\infty}\frac{a(a+1)\cdots (a+k-1)\cdot b(b+1)\cdots (b+k-1)}{k!c(c+1)\cdots(c+k-1)}. $$
Note that if either $a$ or $b$ is negative then the series is finite.

Now, changing the index of the series on the left hand side of the desired identity to $k=i-1$, we let $C_k$ be the $k^{\textrm{th}}$ term. Expanding factorials and cancellation gives
$$\frac{C_{k+1}}{C_k}=\frac{(k+d+2-n)(k+3)}{(n-k-4)(k+1)}.$$
This means that the left hand side of the desired identity can be written as
\begin{align*}
\sum_{k=0}^{n-4}\binom{n-4-k}{d-2}\binom{k+2}{2}=\binom{n-4}{d-2}{}_2F_1\begin{pmatrix}-(n-d-2)&3\\ & 4-n\end{pmatrix}.
\end{align*}
The Chu-Vandermonde Identity~\cite[Equation (2.7)]{Roy1987} gives
the desired identity:
\begin{align*}
\binom{n-4}{d-2}{}_2F_1\begin{pmatrix}-(n-d-2)&3\\ & 4-n\end{pmatrix}&=\binom{n-4}{d-2}\frac{(1-n)(2-n)\cdots(-(d+2))}{(4-n)(5-n)\cdots(-(d-1))}\\
&=\binom{n-4}{d-2}\frac{(n-1)!}{(d+1)!}\frac{(d-2)!}{(n-4)!}\\
&=\frac{(n-1)(n-2)(n-3)(n-4)!}{(d+1)d(d-1)(d-2)!(n-4-d+2)!}\\
&=\binom{n-1}{d+1}. 
\end{align*}
\end{proof}

\section{The number of equations of $\mgbar_{0,n}$ in $\mgbar_{0,n-1}\times\PP^{n-3}$}\label{tools}

We recall cohomological tools developed in~\cite{KeT}. Working in the ideal defining $\mgbar_{0,n}$ as a subscheme of $\mgbar_{0,n-1}\times\PP^{n-3}$, these tools allow us to realize the number of equations of a given bidegree as the dimension of the space of global sections of a certain sheaf on $\mgbar_{0,n-1}$. We apply this to the case $n=5$ to prove that $J_5$ contains all polynomials in $I_5$ of degree $(d,d)$. We do the same for $J_6$ and $I_6$ in the next section.

Let $\sigma: U\to\mgbar_{0,n}$ be the universal curve over $\mgbar_{0,n}$ and $\omega$ the relative dualizing sheaf. The {\em $\kappa$ class} on $\mgbar_{0,n}$ is the pushforward of the first Chern class of $\omega$, i.e. $\kappa=\sigma_*(c_1(\omega))$.
Let $K_{\mgbar_{0,n}}$ be the canonical class on $\mgbar_{0,n}$ and $\delta_I$ the classes of the boundary divisors. In~\cite{KeT}, Keel and Tevelev prove that $\kappa\sim K_{\mgbar_{0,n}}+\sum{\delta_I}$, that $\kappa$ is very ample, and that the composition of $\phi$ with the Segre embedding $\PP^1\times\cdots\times\PP^{n-3}\hookrightarrow \PP^{(n-2)!-1}$ is the embedding of $\mgbar_{0,n}$ via the $\kappa$ class. Proposition~\ref{m05} implies:

\begin{corollary}\label{kappaembedding}The ideal of the embedding of $\mgbar_{0,5}$ into $\mathbb{P}^5$ via the $\kappa$ class is generated by the five quadrics
in (\ref{eq:fivequadrics}).
\end{corollary}

\begin{proof}
The first two equations are given by first multiplying the equation of Proposition~\ref{m05} by $x_0$ and $x_1$ to obtain equations homogeneous of the same degree in $x_i$ and $y_j$, then mapping into $\PP^5$ by the Segre embedding. The final three are the Segre relations.  By \cite{KeT}, the resulting map from $\mgbar_{0,5}\to \PP^5$ is given by the $\kappa$ class. \qed
\end{proof} 

We will prove that for $n=5,6$, the ideals generated by polynomials in $I_n$ of degree $(d,d,\ldots, d)$ are generated by polynomials in $J_n$. The following theorem tells us that this is enough to understand the ideal of $\Phi(\mgbar_{0,n})$ in $\PP^{(n-2)!-1}$ via the $\kappa$ class. 

\begin{theorem}\cite{KeT}\label{quadrics}
The ideal $A_n$ that defines
$\Phi(\mgbar_{0,n})$ as a subscheme of $\PP^{(n-2)!-1}$ is generated by quadrics. 
 Equivalently, let $\tilde{I}_n$ be the ideal generated by all polynomials of degree $(d,d, \ldots, d)$ contained in the ideal $I_n$  in the Cox ring of $\PP^1\times\cdots \times \PP^{n-3}$. Then $\tilde{I}_n$ is generated by polynomials of degree $(2,\ldots,2)$.
\end{theorem}

Let $V_{\psi_n}$ be the vector bundle on $\mgbar_{0,n}$ defined by the exact sequence
\begin{equation}\label{V}
0\to V_{\psi_n}\to H^0(\mgbar_{0,n},\psi_n)\otimes\mathcal{O}_{\mgbar_{0,n}}\to\psi_n\to 0,
\end{equation}
and consider the map
 $\Phi=(\pi_n,\psi_n):\mgbar_{0,n}\to \mgbar_{0,n-1}\times\PP^{n-3}$. By \cite[Lemma 4.1]{KeT}, we have the following resolution of the structure sheaf of $\Phi(\mgbar_{0,n})$ in $\mgbar_{0,n-1}\times\PP^{n-3}$:
\begin{equation}\label{maintool}  \!
 0\to \mathcal{M}_n^{n-4}\boxtimes\mathcal{O}(3{-}n)\to\cdots\to\mathcal{M}_n^{1}\boxtimes\mathcal{O}(-2)\to \mathcal{O}_{\mgbar_{0,n-1}\times\PP^{n-3} }\to \Phi_*\mathcal{O}_{\mgbar_{0,n}} \!\! \to 0,
\end{equation}
where $\mathcal{M}_n^p=R^1\pi_{n*}(\wedge^{p+1}V_{\psi_n})$.

To gain control of the sheaves in~\eqref{maintool}, we have
\begin{theorem}\label{Q} [KT09, Theorem 4.3] There exists a vector bundle $Q$ on $\mgbar_{0,n}$ and exact sequences
\begin{equation}\label{M1}
0\to \pi_n^*\mathcal{M}^p_{n-1}\to \mathcal{M}^p_n\to Q\to 0
\end{equation}
\begin{equation}\label{M2}
0\to V_{\psi_n}\to Q \to \mathcal{M}^{p-1}_n\to 0.
\end{equation}
\end{theorem}

\begin{remark} The vector bundle $Q$ is defined explicitly in \cite{KeT}.
\end{remark}

Thus, by tensoring~\eqref{maintool} with the correct line bundles one can compute the expected number of equations for $\mgbar_{0,n}\subset\mgbar_{0,n}\times\PP^{n-3}$ in particular degrees.


\begin{lemma}\label{numberofequations} For all integers $n\ge5$ and $a>0$, the ideal $I_n$ defining $\mgbar_{0,n}$ as a subscheme of $\mgbar_{0,n-1}\times\PP^{n-3}$ contains exactly $h^0(\mgbar_{0,n-1}, \mathcal{M}^1_{n-1}\otimes\kappa_{n-1}^{\otimes a})$
linearly independent equations of bidegree $(a,2)$. Additionally, we have the short exact sequence 
$$
0\to H^0(\mgbar_{0,n-1},\pi_{n-1}^*\mathcal{M}^1_{n-2}\otimes\kappa_{n-1}^{\otimes a})\to H^0(\mgbar_{0,5},\mathcal{M}^1_{n-1}\otimes\kappa_{n-1}^{\otimes a})\to
$$
$$
\to H^0(\mgbar_{0,n-1},V_{\psi_{n-2}}\otimes\kappa_{n-1}^{\otimes a})\to 0.
$$
\end{lemma}

\begin{proof}
We use the resolution of the structure sheaf of $\Phi_{*}\mathcal{O}_{\mgbar_{0,n}}$ given by the exact sequence~\eqref{maintool}. Tensoring this with $\kappa_{n-1}^{\otimes a}\boxtimes\mathcal{O}(2)$, we obtain the following exact sequence on $\mgbar_{0,n-1}\times\PP^{n-3}$:
$$
0\to(\mathcal{M}_{n-1}^{n-4}\otimes\kappa_{n-1}^{\otimes a})\boxtimes\mathcal{O}(1-n)\to(\mathcal{M}_{n-1}^{n-3}\otimes\kappa_{n-1}^{\otimes a})\boxtimes\mathcal{O}(2-n)\to \cdots 
$$
$$
\to (\mathcal{M}^1_{n-2}\otimes\kappa_{n-1}^{\otimes a})\boxtimes\mathcal{O}(-1)\to (\mathcal{M}^1_{n-1}\otimes\kappa_{n-1}^{\otimes a})\boxtimes\mathcal{O}\to\kappa_{n-1}^{\otimes a}\boxtimes\mathcal{O}(2)\to
$$
$$
\to \Phi_{*}\mathcal{O}_{\mgbar_{0,n}}\otimes(\kappa^{\otimes a}\boxtimes\mathcal{O}(2))\to0.
$$
By the K\"unneth Formula, and since $\mathcal{O}(k)$ is acyclic for $1-n\le k\le-1$, we have 
$$H^i(\mgbar_{0,n-1}\times\PP^{n-3}, (\mathcal{M}_{n-1}^{k+2n-5}\otimes\kappa_{n-1}^{\otimes a})\boxtimes\mathcal{O}(k))=0$$
for all $1-n\le k\le-1$ and $i\ge0$.  Moreover, by Lemma 6.5 of \cite{KeT}, we have 
$$H^1(\mgbar_{0,n-1}\times\PP^{n-3}, (\mathcal{M}^1_{n-1}\otimes\kappa_{n-1}^{\otimes a})\boxtimes\mathcal{O})=0.$$
Thus, we obtain a short exact sequence in cohomology:
$$
0\to H^0(\mgbar_{0,n-1}\times\PP^{n-3}, (\mathcal{M}^1_{n-1}\otimes\kappa_{n-1}^{\otimes a})\boxtimes\mathcal{O})\to H^0(\mgbar_{0,n-1}\times\PP^{n-3},\kappa_{n-1}^{\otimes a}\boxtimes\mathcal{O}(2))\xrightarrow{\tau}
$$
$$
\xrightarrow{\tau}H^0(\mgbar_{0,n-1}\times\PP^{n-3},\Phi_{*}\mathcal{O}_{\mgbar_{0,n}}\otimes(\kappa^{\otimes a}\boxtimes\mathcal{O}(2)))\to 0.
$$
In particular, the number of equations of bidegree $(a,2)$ in the ideal defining $\mgbar_{0,n}$ as a subvariety of $\mgbar_{0,n-1}\times \PP^{n-3}$ is given by the dimension of the kernel of $\tau$, which is
$H^0(\mgbar_{0,n-1}\times\PP^{n-3}, (\mathcal{M}^1_{n-1}\otimes\kappa_{n-1}^{\otimes a})\boxtimes\mathcal{O})$.
By the K\"unneth Formula, this vector space is isomorphic to 
$H^0(\mgbar_{0,n-1}, \mathcal{M}^1_{n-1}\otimes\kappa_{n-1}^{\otimes a})$.

Using the short exact sequences~\eqref{M1} and~\eqref{M2}, and noting that $\mathcal{M}^0_n=0$ for all $n$, we have the short exact sequence
$$0\to \pi_{n-1}^*\mathcal{M}^1_{n-2}\to\mathcal{M}^1_{n-1}\to V_{\psi_{n-2}}\to 0.$$
We tensor with $\kappa_{n-1}^{\otimes a}$. Then 
$H^i(\mgbar_{0,n-1},\pi_{n-1}^*\mathcal{M}^1_{n-2}\otimes\kappa_{n-1}^{\otimes a})=0$
 for $i>0$, by Lemma 6.5 in \cite{KeT}, and so taking cohomology gives the short exact sequence. \qed
 \end{proof} 

\begin{example}
For a reality check, recall from Proposition~\ref{m05} that $I_5$ is generated by one equation of degree $(1,2)$, and Lemma~\ref{all} gives one equation 
satisfied by $\mgbar_{0,5}$ in $\PP^1\times\PP^2$.
On the other hand, Lemma~\ref{numberofequations} tells us that $\mgbar_{0,5}$ has 
$h^0(\mgbar_{0,4},\mathcal{M}^1_4\otimes\kappa_4^{\otimes a})$  linearly independent equations of degree $(a,2)$. 
\end{example}

\begin{theorem} We have
$h^0(\mgbar_{0,4},\mathcal{M}^1_4\otimes\kappa_4)=1$ and
$h^0(\mgbar_{0,4},\mathcal{M}^1_4\otimes\kappa_4^{\otimes2})=2$.
\end{theorem}

\begin{proof}
Under the isomorphism $\mgbar_{0,4}\simeq \PP^1$, we have $\kappa_4=\mathcal{O}_{\PP^1}(1)$. We show that on $\mgbar_{0,4}$, we have $\mathcal{M}^1_4=\psi_4^{-1}=\mathcal{O}_{\PP^1}(1)$.
Since $\mathcal{M}^1_3=0$, the exact sequence \eqref{M1} becomes
$$
0\to 0 \to \mathcal{M}^1_4\to Q\to 0
$$
and tells us that $\mathcal{M}^1_4 \simeq  Q$. Since $ \mathcal{M}^{0}_n=0$ for all $n$, the sequence \eqref{M2} becomes
$$
0\to V_{\psi_4}\to Q \to0\to 0,
$$
and therefore $\mathcal{M}^1_4 \simeq Q \simeq V_{\psi_4}$. We can determine the bundle $V_{\psi_4}$ by analyzing the exact sequence~\eqref{V}. On $\mgbar_{0,4}$, this is
$$
0\to V_{\psi_4}\to H^0(\mgbar_{0,4},\psi_4)\otimes\mathcal{O}_{\mgbar_{0,4}}\to\psi_4\to 0.
$$
Taking determinants gives the following equality of line bundles:
$$
\mathcal{O}_{\mgbar_{0,4}}=\det\left(H^0(\mgbar_{0,4},\psi_4)\otimes\mathcal{O}_{\mgbar_{0,4}}\right)=\det(V_{\psi_4})\otimes\det(\psi_4)
=V_{\psi_4}\otimes\psi_4.
$$
In particular, $\mathcal{M}^1_4 \simeq V_{\psi_4}$ is a line bundle dual to $\psi_4$. Recall that $\mgbar_{0,4}\simeq \PP^1$ and $\psi_4 \simeq \mathcal{O}_{\PP^1}(1)$, so the dimension of $H^0(\mgbar_{0,4},\mathcal{M}_4^1\otimes\mathcal{O}_{\PP^1}(2))$ is given by 
$$
h^0(\mgbar_{0,4},\mathcal{M}_4^1\otimes\mathcal{O}_{\PP^1}(2)) =h^0(\PP^1,\mathcal{O}_{\PP^1}(-1)\otimes\mathcal{O}_{\PP^1}(2))=h^0(\PP^1,\mathcal{O}_{\PP^1}(1))=2.
$$
Similarly we compute the dimension of $H^0(\mgbar_{0,4},\mathcal{M}_4^1\otimes\mathcal{O}_{\PP^1}(1))$ to be
$$
h^0(\mgbar_{0,4},\mathcal{M}_4^1\otimes\mathcal{O}_{\PP^1}(1)) =h^0(\PP^1,\mathcal{O}_{\PP^1}(-1)\otimes\mathcal{O}_{\PP^1}(1))=h^0(\PP^1,\mathcal{O}_{\PP^1})=1.
$$
\end{proof}

\section{Equations for $\mgbar_{0,6}$}\label{for6}

We apply the tools from the previous section to prove Theorem~\ref{m06}, which states that $\tilde{I}_6$ is generated by polynomials in $J_6$. 

Lemma \ref{all} gives five polynomials satisfied by 
 $\mgbar_{0,6}$ in $\PP^1\times\PP^2\times\PP^3$:
  $$ \begin{matrix}
f_1  =  b_1c_1c_2-b_2c_1c_2+b_2c_1c_3-b_1c_2c_3, &
f_2  =  b_0c_0c_2-b_2c_0c_2+b_2c_0c_3-b_0c_2c_3, \\
f_3  =  b_0c_0c_1-b_1c_0c_1+b_1c_0c_3-b_0c_1c_3, &
f_4  =  a_0c_0c_1-a_1c_0c_1+a_1c_0c_2-a_0c_1c_2, \\
f_5 =  a_0b_0b_1-a_1b_0b_1+a_1b_0b_2-a_0b_1b_2.&
\end{matrix} $$
Let $J_6$ be the ideal in $\mathbb{C}[a_0,a_1,b_0,b_1,b_2,c_0,c_1,c_2,c_3]$ generated by
  $f_1,\ldots, f_5$, and let $I_6$ be the unique $B$-saturated ideal 
  defining $\mgbar_{0,6}$ scheme-theoretically, 
  where 
$$B=\langle a_0,a_1\rangle\cap\langle b_0,b_1,b_2\rangle\cap\langle c_0,c_1,c_2,c_3\rangle.$$
Using \emph{Macaulay2}, we verified that $I_6$ is prime, and is generated by  $J_6$ and
$$f_6\,\,=\,\, a_0b_0c_1c_2-a_0b_2c_1c_2-a_0b_0c_1c_3+a_1b_0c_1c_3+a_0b_2c_1c_3-a_1b_0c_2c_3.$$

  \begin{proposition}\label{M2dim}
 The ideal $J_6$ is properly contained in the ideal $I_6$, but the parts of homogeneous degrees $(2,2,2)$ of the ideals $J_6$ and $I_6$ coincide.
 \end{proposition}

\begin{proof}
Let $I_{(i,j,k)}$, respectively $J_{(i,j,k)}$, be the vector space of polynomials of degree $(i,j,k)$ in $I_6$, respectively $J_6$. Multiplying the polynomials $f_1,\ldots,f_5$ by all monomials of the correct degree gives a spanning set of $J_{(i,j,k)}$. Computing the dimension of $J_{(i,j,k)}$ involves determining which of the resulting polynomials are redundant. We used \emph{Macaulay2} to show that
$ \dim J_{(1,1,2)}=9 < \dim I_{(1,1,2)}=10$ and $\dim J_{(2,2,2)}= \dim I_{(2,2,2)}=55$. 
\qed
\end{proof}

\begin{remark}
As we will see, the second part of Proposition~\ref{M2dim} implies that the  homogeneous parts of the ideals $I_6$ and $J_6$ coincide. Corollary~\ref{itworks}  shows that these ideals contain the correct number of homogeneous equations of degree $(2,2,2)$.
\end{remark}

\begin{theorem}\label{m06}
Let $\tilde{I}_6$ be the ideal generated by the polynomials of degree $(d,d,d)$ in $I_6$. Then $\tilde{I}_6$ is generated by the polynomials of degree $(2,2,2)$ contained in $J_6$. Equivalently, the embedding $\Phi(\mgbar_{0,n})$ in $\PP^{23}$ defined by the $\kappa$ class is generated by the homogeneous polynomials of degree $(2,2,2)$ in $J_6$ and the Segre relations.
\end{theorem}

The proof of Theorem~\ref{m06} requires the following two lemmas.

\begin{lemma}\label{onm051} We have
$h^0(\mgbar_{0,5},V_{\psi_4}\otimes\kappa_5^{\otimes 2})=24$ and
$h^0(\mgbar_{0,5},V_{\psi_4}\otimes\kappa_5)=11$.
\end{lemma}
\begin{proof}
We have on $\mgbar_{0,5}$ the short exact sequence
$$0\to V_{\psi_4}\to \mathbb{C}^3\otimes\mathcal{O}_{\mgbar_{0,5}}\to \psi_4\to 0.$$
We tensor with $\kappa_5^{\otimes 2}$. 
Noting that $H^i(\mgbar_{0,5}, V_{\psi_4}\otimes\kappa_5^{\otimes 2})=0$ for $i>0$ 
by \cite[Lemma 6.5]{KeT}, the long exact sequence in cohomology gives the short exact sequence
$$0\to H^0(\mgbar_{0,5},V_{\psi_4}\otimes\kappa_5^{\otimes 2})\to H^0(\mgbar_{0,5},\mathbb{C}^3\otimes\kappa_5^{\otimes 2})\to H^0(\mgbar_{0,5},\psi_4\otimes\kappa_5^{\otimes 2})\to 0.$$
We shall determine the dimensions of the middle term and the last term.

For the middle term, we note first that any global section of $\mathbb{C}^3\otimes\kappa_5^{\otimes 2}$ is of the form $\alpha\otimes \beta$ where $\alpha\in \mathbb{C}^3$ and $\beta\in H^0(\mgbar_{0,5},\kappa_5^{\otimes 2})$. Using that $\mgbar_{0,5}\simeq \textrm{Bl}_{q_1,\ldots, q_4}\PP^2$, we let $\sigma:\textrm{Bl}_{q_1,\ldots, q_4}\PP^2\to \PP^2$ be the blowup, $E_1,\ldots,E_4$ the exceptional divisors, and $L_{i,j}$ the proper transform of the line passing through $q_i$ and $q_j$ on $\PP^2$. The $\kappa$ class is given by $\kappa_5= K_{\mgbar_{0,5}}+\sum \delta_I$, where the sum is taken over all boundary divisors $\delta_I$ of $\mgbar_{0,5}$, so we have the linear equivalence
$$
\kappa_5\sim \left(\sigma^*(K_{\PP^2})+\sum_{i=1}^4E_i\right) + \left(\sum_{i=1}^4E_i+\sum_{1\le i< j\le4} L_{i,j}\right).
$$

Since $L_{i,j}\sim \sigma^*{H}-E_i-E_j$, where $H$ is the class of a hyperplane section on $\PP^2$, we can write $\kappa_5^{\otimes 2}$ as
$\,
(\kappa_5)^{\otimes 2}\sim\sigma^*(6H)-2\sum_{i=1}^4 E_i$.
In particular, this gives
$$
H^0(\mgbar_{0,5},\kappa_5^{\otimes 2})\,\,\simeq \,\, H^0\left(\PP^2, \mathcal{O}_{\PP^2}(6H)\bigotimes_{i=1}^4\mathcal{I}_{q_i}^{\otimes2}\right),
$$
where $\mathcal{I}_{q}$ denotes the skyscraper sheaf at $q$.
By explicitly writing out equations, one can check that the conditions that a curve $C$ of degree $6$ on $\PP^2$ attain nodes at four points in general position are linearly independent. Thus, this latter vector space has dimension $\binom{8}{2}-12=16$, and
therefore
$
h^0(\mgbar_{0,5},\mathbb{C}^3\otimes\kappa_5^{\otimes 2})=3\cdot16=48.
$

Repeating the above with $\kappa_5^{\otimes 2}$ replaced by $\kappa_5$, we have $h^0(\mgbar_{0,5},\mathbb{C}^3\otimes\kappa_5)=3
\begin{pmatrix} \binom{5}{2}-4 \end{pmatrix} =18.$

Next, since $\psi_4\sim \delta_{1,2}+\delta_{3,5}+\delta_{4,5}\sim H$, we have
$
\psi_4\otimes\kappa_5^{\otimes2}\,\sim\, \sigma^*(7H)-2 (E_1+E_2+E_3+E_4)$.
With this, we compute the dimension of $H^0(\mgbar_{0,5},\psi_4\otimes \kappa_5^{\otimes 2})$ to be
$$
h^0(\mgbar_{0,5},\psi_4\otimes \kappa_5^{\otimes 2})=h^0\left(\PP^2,\mathcal{O}_{\PP^2}(7)\bigotimes_{i=1}^4\mathcal{I}_{q_i}^{\otimes 2}\right)=\binom{9}{2}-12=24.
$$

Finally, we repeat the above computation with $\kappa_5^{\otimes 2}$ replaced by $\kappa_5$, 
and we find
$$ \qquad \qquad
h^0(\mgbar_{0,5},\psi_4\otimes \kappa_5)=h^0\left(\PP^2,\mathcal{O}_{\PP^2}(4)\bigotimes_{i=1}^4\mathcal{I}_{q_i}\right)=\binom{6}{2}-4=11. \qquad \qquad \qed
$$
\end{proof} 

\begin{lemma}\label{onm052} We have
$\,h^0(\mgbar_{0,5},\pi^*\mathcal{M}^1_{4}\otimes\kappa_5^{\otimes 2})=11\,$ and
$\,h^0(\mgbar_{0,5},\pi^*\mathcal{M}^1_{4}\otimes\kappa_5)=3$.
\end{lemma}

\begin{proof}
We note first that on $\mgbar_{0,4}\simeq\PP^1$, we have $\mathcal{M}^1_4=\psi_4^{-1}$ and $\psi_4=\mathcal{O}(-1)$, 
so $\pi^*(\mathcal{M}^1_{4})=\pi^*(\psi_4^{-1})$. Thus, 
$\,\pi^*(\mathcal{M}^1_4)\,=\,\pi^*(\mathcal{O}(1))=-2H+\sum_{i=1}^4E_i$.
This allows us to write the class $\pi^*\mathcal{M}^1_4\otimes\kappa_5^{\otimes 2}$ as 
$$
\pi^*\mathcal{M}^1_{4}\otimes\kappa_5^{\otimes 2}\sim \sigma^*(-2H)+\sum_{i=1}^4E_i+ \sigma^*(6H)-2\sum_{i=1}^4 E_i=\sigma^*(4H)-\sum_{i=1}^4 E_i.
$$
Using the push-pull formula, we can now compute the desired dimensions:
$$
h^0(\mgbar_{0,5},\pi^*\mathcal{M}^1_{4}\otimes\kappa_5^{\otimes 2})=h^0(\PP^2,  \mathcal{O}_{\PP^2}(4) \otimes \bigotimes_{i=1}^4\mathcal{I}_{q_i})=\binom{6}{2}-4=11.
$$
$$  \qquad \qquad \qquad
h^0(\mgbar_{0,5},\pi^*\mathcal{M}^1_{4}\otimes\kappa_5)=h^0(\PP^2,  \mathcal{O}_{\PP^2}(1))=3. \qquad \qquad \qquad \qed
$$
\end{proof} 

Applying Lemmas~\ref{onm051} and~\ref{onm052} to the short exact sequence of Lemma~\ref{numberofequations} gives 

\begin{corollary}\label{itworks}
The number of equations defining $\mgbar_{0,6}$ in the line bundle $\kappa_5^{\otimes 2}\boxtimes \mathcal{O}(2)$ on $\mgbar_{0,5}\times\PP^3$ is $35$.
Similarly,  for the line bundle $\kappa_5\boxtimes \mathcal{O}(2)$,
the number is $10$.
\end{corollary}

\begin{proof}
For $n=6$ and $a=1,2$, Lemma~\ref{numberofequations} gives the short exact sequences
$$0\to H^0(\mgbar_{0,5},\pi^*\mathcal{M}^1_{4}\otimes\kappa_5^{\otimes 2})\to H^0(\mgbar_{0,5},\mathcal{M}^1_5\otimes\kappa_5^{\otimes 2})\to H^0(\mgbar_{0,5},V_{\psi_4}\otimes\kappa_5^{\otimes 2})\to 0$$
$$0\to H^0(\mgbar_{0,5},\pi^*\mathcal{M}^1_{4}\otimes\kappa_5)\to H^0(\mgbar_{0,5},\mathcal{M}^1_5\otimes\kappa_5)\to H^0(\mgbar_{0,5},V_{\psi_4}\otimes\kappa_5)\to 0.$$
The result follows from these, and Lemmas~\ref{onm051} and~\ref{onm052}. \qed
\end{proof} 


\begin{proof}[of Theorem~\ref{m06}]
There are no equations in $I_5$ of degree $(1,1)$, so the ideal generated by equations of degree $(1,1,2)$ in $I_6$ has number of linearly independent polynomials coinciding with the number of linearly independent polynomials defining $\mgbar_{0,6}$ in the line bundle $\kappa_5\boxtimes \mathcal{O}(2)$ on $\mgbar_{0,5}\times\PP^3$.

Since there are two equations in $I_5$ of degree $(2,2)$ in the Cox ring of $\PP^1\times\PP^2$, both of which must be homogenized, we see that $I_6$ contains 20 linearly independent polynomials of degree $(2,2,2)$ in the Cox ring of $\PP^1\times\PP^2\times\PP^3$. 
By Corollary~\ref{itworks}, the number of linearly independent 
polynomials of degree $(2,2,2)$ is $
h^0(\mgbar_{0,5},\mathcal{M}^1_5\otimes\kappa_5^{\otimes 2})+20=55$. 
The number of linearly independent polynomials in $I_6$ of degree $(1,1,2)$ equals $10$.
Together with Proposition~\ref{M2dim}, this completes the proof.\qed
\end{proof}





\section{Appendix}

We include the \emph{Macaulay2} code used to verify Conjecture~\ref{conj:1}.
\begin{verbatim}
R = QQ[a0,a1,b0,b1,b2,c0,c1,c2,c3,d0,d1,d2,d3,d4,
	e0,e1,e2,e3,e4,e5,f0,f1,f2,f3,f4,f5,f6];
\end{verbatim}
The rows of the following matrix are coordinates on $\PP^1,\PP^2,\PP^3, \PP^4,\PP^5,\PP^6$
\begin{verbatim}
M=matrix{{0,0,0,0,0,0,0},{a0,a1,0,0,0,0,0},
	{b0,b1,b2,0,0,0,0},{c0,c1,c2,c3,0,0,0},
	{d0,d1,d2,d3,d4,0,0},{e0,e1,e2,e3,e4,e5,0},
	{f0,f1,f2,f3,f4,f5,f6}};

M05 = {{2,2}};  M06 = {{2,2}, {3,2},{3,3}};
M07 = {{2,2}, {3,2},{3,3}, {4,2},{4,3},{4,4}};
M08 = {{2,2}, {3,2},{3,3}, {4,2},{4,3},{4,4},{5,2},
	{5,3},{5,4},{5,5}};
M09 = {{2,2}, {3,2},{3,3}, {4,2},{4,3},{4,4}, {5,2},
	{5,3},{5,4},{5,5}, {6,2},{6,3},{6,4},{6,5},{6,6}};
\end{verbatim}
Select your desired $n$ here:
\begin{verbatim}
L = M07;
\end{verbatim}
Lemma~\ref{all} involves the $2\times2$-minors of the matrices
\begin{verbatim}
Q=apply(L,l->{submatrix(M,{l_1-1,l_0},0..l_1-1),
			M_(l_1)_(l_0)})
S=apply(Q,T-> matrix{apply(entries transpose T_0,
			x->x_0*(x_1-T_1)),(entries T_0)_1})
\end{verbatim}
We form the ideal $J=J_n$ of all such $ 2 \times 2$-minors, and compute the prime ideal $I$ by saturation.
\begin{verbatim}
J = sum apply(S,N -> minors(2,N));
J = saturate(J,ideal(a0,a1));
J = saturate(J,ideal(b0,b1,b2));
J = saturate(J,ideal(c0,c1,c2,c3));
J = saturate(J,ideal(d0,d1,d2,d3,d4));
J = saturate(J,ideal(e0,e1,e2,e3,e4,e5));
I = saturate(J,ideal(f0,f1,f2,f3,f4,f5,f6));
\end{verbatim}
The following are used to determine whether the dimension of $I$ is correct, as well as compute the degree of $I$, and the minimal number of generators in each degree.
\begin{verbatim}
codim I, degree I, betti mingens I
\end{verbatim}
We finally note that the initial ideal is square-free and Cohen-Macaulay:
\begin{verbatim}
M = monomialIdeal leadTerm I;
betti mingens M
\end{verbatim}

\begin{acknowledgement}
This article was initiated during the Apprenticeship Weeks (22 August-2 September 2016), led by Bernd Sturmfels, as part of the Combinatorial Algebraic Geometry Semester at the Fields Institute. We thank Bernd Sturmfels for providing inspiration and feedback on multiple drafts. We are also grateful to Christine Berkesch Zamaere, Renzo Cavalieri, Diane Maclagan, Steffen Marcus, Vic Reiner, and Jenia Tevelev for many helpful discussions. The second author was partially supported by a scholarship from the Clay Math Institute.
\end{acknowledgement}

\end{document}